\documentclass[12pt]{article}
\usepackage{amsmath,amsfonts,amsthm,amssymb,mathrsfs,mathtools,ulem}
\usepackage{graphicx}
\usepackage[usenames,dvipsnames]{color}
\usepackage{float}
\usepackage{graphicx}
\usepackage{subfigure}
\usepackage{caption}
\usepackage{cite}
\usepackage{url}
\usepackage{caption}
\usepackage{epstopdf}
\usepackage{hyperref}
\usepackage[latin1]{inputenc}
\usepackage{color}

\newcommand{\dint}{\displaystyle\int}

\newcommand\redout{\bgroup\markoverwith
{\textcolor{red}{\rule[0.5ex]{2pt}{0.8pt}}}\ULon}

\theoremstyle{plain}
\newtheorem{theorem}{Theorem}[section]
\newtheorem{hy}{Assumption}[section]
\newtheorem{corollary}[theorem]{Corollary}
\newtheorem{lemma}[theorem]{Lemma}
\newtheorem{proposition}[theorem]{Proposition}

\theoremstyle{definition}
\newtheorem{definition}[theorem]{Definition}

\theoremstyle{remark}
\newtheorem{remark}[theorem]{Remark}

\numberwithin{equation}{section}
\numberwithin{theorem}{section}

\usepackage{a4wide}
\usepackage{geometry}
\begin{document}

\title{Bridges with random length: \\ Gamma case}

\author{\textbf{Mohamed Erraoui}\\
	Universit{\'e} Cadi Ayyad, Facult{\'e} des Sciences Semlalia,\\
	D{\'e}partement de Math{\'e}matiques, B.P. 2390, Marrakech, Maroc\\
	Email: erraoui@uca.ac.ma 
	\and
	\textbf{Astrid Hilbert}\\
		School of Computer Science, Physics and Mathematics,\\
		Linnaeus University, Vejdesplats 7, SE-351 95 V\"axj\"o, Sweden.\\
		Email: astrid.hilbert@lnu.se 
		\and 
\textbf{Mohammed Louriki}\\
Universit{\'e} Cadi Ayyad, Facult{\'e} des Sciences Semlalia,\\
 D{\'e}partement de Math{\'e}matiques, B.P. 2390, Marrakech, Maroc\\
 Email: louriki.1992@gmail.com}

\maketitle
\begin{abstract}
	The aim objective of this paper is to show that certain basic properties of  gamma bridges with deterministic length stay true also for gamma bridges with random length. Among them the Markov property as well as the canonical decomposition with respect to the usual augmentation of its natural filtration, which leads us to conclude that its completed natural filtration is right continuous.
\end{abstract}

\\
\\
\\
\\

\begin{center}
\section{Introduction}
\label{Setion_1}
\end{center}
The gamma process has proven very successful when modelling accumulation processes. Early studies by e.g. Hammersley \cite{H}, Moran \cite{M}, Gani \cite{G}, Kendall \cite{K}, Kingman \cite{Ki} addressed the modelling of water stored in and released from reservoirs and accumulation related to storage in general. Dufresne et al. \cite{DGS} show how to employ the gamma process to model liabilities of insurance portfolios for continuous claims. For these risk models the fixed budget horizon of one year is assumed. The authors investigate furthermore the gamma process in the setting of ruin theory and supply ruin probabilities in form of tables. The gamma process replaces the compound processes used traditionally. In Em\'ery and Yor \cite{EY} and Yor \cite{Y} gamma bridges were studied and their application to stop loss reinsurance and credit risk management were pointed out in Brody et al \cite{BHM}. This work introduces and focusses on random gamma bridges, which model accumulated losses of large credit portfolios in credit risk management. These studies were continued by Hoyle et al. in a series of papers, see e.g. Hoyle and Meng\"ut\"urk \cite{HM}. In a further article pricing at an intermediate time is studied Hoyle et al \cite{HHM}. Returning to the starting point accumulation processes for storage, we refer to recent developments in Chan et al \cite{CTT} for further references. Numerical results may be found in Assmusen and Hobolth \cite{AH}.\\
In this paper we generalize the concept of a gamma bridge to random times, at which the bridge is pinned, to study amongst others its Markov property and to give its decomposition semi-martingale. There are two recent works in which bridges with random length are studied. The first by Bedini et al \cite{BBE} studies related properties of the Brownian bridge with random length, the second by Erraoui and Louriki \cite{EL} studies Gaussian bridges with random length. In both works existence of an explicit expression for the bridge with random length is exploited. Applications for the random gamma bridge suggest itself for accumulation processes in financial mathematics, see the results of Jeanblanc and Le Cam in \cite{JC2010} and \cite{JC2009}, if heavier tails are desired, electricity production in river power plants with implications to electricity prices, life insurance portfolios, where the instant of death is random, accumulation storage in the setting of just in time production.\\
The paper is organized as follows. Section 2 begins by recalling the definitions and some properties of gamma processes and
gamma bridges of deterministic length, which will be used throughout the paper. In Section 3 we define the gamma bridge with random time $\tau$ which will be denoted by $\zeta$ and we consider the stopping time property of $\tau$ with respect to the right continuous and completed filtration $\mathbb{F}^{\zeta,c}_+$  generated by the process $\zeta$. Moreover, we give the conditional distribution of $\tau$ and $\zeta_u$ given $\zeta_{t}$ for $u>t>0$. Next we establish the Markov property of the process $\zeta$ with respect to its completed natural filtration.  As a consequence, we derive Bayesian estimates for the distribution of the default time $\tau$, given the past behaviour of the process $\zeta$ up to time $t$. After that we study the Markov property of the gamma bridge with random length, with respect to $\mathbb{F}^{\zeta,c}_+$. Finally we give its semimartingale decomposition with respect to $\mathbb{F}^{\zeta,c}_+$.\\
The following notation will be used throughout the paper: 
For a complete probability space $(\Omega,\mathcal{F},\mathbb{P})$, $\mathcal{N}_p$ denotes the
collection of $\mathbb{P}$-null sets. If $\theta$ is a random variable, then $\mathbb{P}_{\theta}$
denotes the law of $\theta$ under $\mathbb{P}$. $\mathcal{D}$ denotes the space of right continuous functions with
left limits (càdlàg) from $\mathbb{R}_{+}$ to $\mathbb{R}_{+}$,
endowed with Skorohod's topology, under which the space
$\mathcal{D}$ is a Pollish space. If $E$ is a topological space, then the Borel $\sigma$-algebra over $E$ will be denoted by $\mathcal{B}(E)$. The characteristic function of a set $A$ is written $\mathbb{I}_{A}$. The symmetric difference of two sets $A$ and $B$ is denoted by $A\Delta B$. 
Finally for any process $Y=(Y_t,\, t\geq 0)$ on $(\Omega,\mathcal{F},\mathbb{P})$, we define by:
\begin{enumerate}
	\item[(i)] $\mathbb{F}^{Y}=\bigg(\mathcal{F}^{Y}_t:=\sigma(Y_s, s\leq t),~ t\geq 0\bigg)$ the natural filtration of the process $Y$.
	\item[(ii)] $\mathbb{F}^{Y,c}=\bigg(\mathcal{F}^{Y,c}_t:=\mathcal{F}^{Y}_t\vee \mathcal{N}_{P},\, t\geq 0\bigg)$ the completed natural filtration of the process $Y$.
	\item[(iii)] $\mathbb{F}^{Y,c}_{+}=\bigg(\mathcal{F}^{Y,c}_{t^{+}}:=\underset{{s>t}}\bigcap\mathcal{F}^{Y,c}_{s}=\mathcal{F}^{Y}_{t^{+}}\vee \mathcal{N}_{P},\, t\geq 0\bigg)$ the smallest filtration containing $\mathbb{F}^{Y}$ and satisfying the usual hypotheses of right-continuity and completeness.
\end{enumerate}
\begin{center}
	\section{Gamma and Gamma bridge processes}
\end{center}
The purpose of this section is to  recall the definition and some properties of the  standard gamma process and the gamma bridge with deterministic length.\\
\subsection{Gamma process}
       By a standard gamma process $(\gamma_{t},t\geq0)$ on $(\Omega,\mathcal{F},\mathbb{P})$,
we mean a subordinator without drift having the Lévy-Khintchine representation
given by

\begin{eqnarray}
\mathbb{E}(\exp(-\lambda\gamma_{t}))  &= & \exp\left(-t\int_{0}^{\infty}(1-\exp(-\lambda x))\,\dfrac{\exp(-x)}{x}dx\right) \label{eq:levy1}
\\ \nonumber
\\
 & = & (1+\lambda)^{-t}, \label{eq:levy2}
\end{eqnarray}
where $\nu(dx)=\dfrac{\exp(-x)}{x}\,\mathbb{I}_{\left(0,\infty\right)}(x)\,dx$
is the so-called Lévy measure. We note that the formula \eqref{eq:levy2} is obtained
from \eqref{eq:levy1} using the Frullani formula.

The following properties, inferred from \eqref{eq:levy2} by means
of standard arguments (see, e.g., Sato \cite{SA}, ch. 2 and 4), describe the paths of the gamma process.
\begin{proposition}
The gamma process $(\gamma_{t},t\geq0)$ has the following properties: 
\end{proposition}

\begin{description}
\item[(i)] $\gamma$ is a purely jump process;
\item[(ii)] $\gamma$ is not a compound Poisson process and its jumping times
are countable and dense in $[0,\infty)$ a.s.; 
\item[(iii)] the map $t\mapsto\gamma_{t}$ is strictly increasing and not continuous
anywhere a.s.; 
\item[(iv)] $\gamma$ has sample paths of finite variation a.s.; 
\item[(v)] $\gamma_{t},t>0$, follows a gamma distribution with density
\begin{equation}
f_{\gamma_{t}}(x)=\dfrac{x^{t-1}\exp(-x)}{\Gamma(t)}\mathbb{I}_{\left(0,\infty\right)}(x),\label{equationgammadensity}
\end{equation}
 where $\Gamma$ is the gamma function. 
\end{description}
The second property means that, for any $t>0$, $\gamma$ has infinite
activity, that is, almost all paths have infinitely many jumps along
any time interval of finite length. It is a direct consequence of
$\nu(\mathbb{R}_{+})=+\infty$, whereas the fourth property arises from $\dint_{0}^{1}x\,\nu(dx)<+\infty$.

\begin{remark}
\begin{enumerate}

\item[1.] It is clear that the process gamma
$(\gamma_{t},t\geq0)$ is a process with paths in $\mathcal{D}$.

\item[2.] The process $(\gamma_{t}-\gamma_{t^{-}}:=e_{t},t\geq0)$ of jumps
of the gamma process $(\gamma_{t},t\geq0)$ is a Poisson point process
whose intensity measure is the Lévy measure of $(\gamma_{t},t\geq0)$,
see Bertoin \cite{B}. For $r>0$, let us denote by $\left(J_{1}^{r}\geq J_{2}^{r}\geq\ldots\right)$
the sequel of the lengths of jumps of the process $(\gamma_{t},t\in [0,r])$
ranked in decreasing order. It is not difficult to see that since
the intensity measure of the Poisson point process $\left((t,e_{t}),t\geq0)\right)$
is $dt\;\dfrac{\exp(-x)}{x}\,\mathbb{I}_{\left(0,\infty\right)}(x)\,dx$,
then the jump times $\left(U_{1}^{r},U_{2}^{r},\ldots\right)$ constitute
a sequence of i.i.d r.v.\textquoteright s with uniform law on $[0,r]$
which is independent from the sequence $\left(J_{k}^{r},k\geq1\right)$. Thus we have the following representation:
\begin{equation}
\gamma_{t}=\underset{k\geq1}{\sum}\,J_{k}^{r}\,\mathbb{I}_{\{ U_{k}^{r}\leq t\}},\quad t\in[0,r].\label{gammarep}
\end{equation}
We note that : $\gamma_{r}=\underset{k\geq1}{\sum}\,J_{k}^{r}$.
\end{enumerate}
\end{remark}

The next proposition gives three other useful properties of of the gamma process.
\begin{proposition}
\label{propstrongpropertiesofgamma}  
\begin{description}
\item[(i)] For every $r>0$, the $\sigma$-algebras, $\sigma\big(\dfrac{\gamma_{u}}{\gamma_{r}},u\in[0,r]\big)$
and $\sigma(\gamma_{u},u\in[r,\infty))$ are independent. 
\item[(ii)] For any $r>0$, $(\gamma_{t},0\leq t\leq r)$ satisfies the following
equation
\begin{equation}
\gamma_{t}=M_{t}^{r}+\int_{0}^{t}\dfrac{\gamma_{r}-\gamma_{s}}{r-s}ds,\label{gammarepmart}
\end{equation}
 where $(M_{t}^{r},t\in[0,r])$ is a $\mathcal{G}_{t}^{(r)}$-martingale
with $\mathcal{G}_{t}^{(r)}=\sigma(\gamma_{s},s\in[0,t]\cup\{r\})$. 
\item[(iii)] $(\gamma_{t},t\geq0)$ has the Markov property with respect to its natural filtration.
\end{description}
\end{proposition}
\begin{proof}
For (i) and (ii) See, \cite{EY}. (iii) $(\gamma_{t},t\geq0)$ has the Markov property since it is a Lévy process.
\end{proof}

For a deeper investigation on the properties of the gamma process
we refer to Kyprianou \cite{KY}, Sato \cite{SA} and Yor \cite{Y}. 
\subsection{Gamma bridge with deterministic length}
		A bridge is a stochastic process that is pinned to some fixed point at a fixed future time. In this section we define the gamma bridge with deterministic length and we give some important properties of this process.
For fixed $r>0$, we define the gamma bridge of length $r$ by setting
\begin{definition}
Let $r\in(0,+\infty)$. The map $\zeta^{r}:\Omega\longmapsto\mathcal{D}$
defined by 
\begin{equation}
\zeta_{t}^{r}(\omega):=\dfrac{\gamma_{t\wedge r}(\omega)}{\gamma_{ r}(\omega)},~t\geq0,~\omega\in\Omega,\label{rgammabridge}
\end{equation}
is the bridge associated with the standard gamma process $(\gamma_{t},t\geq0)$.
Then clearly $\zeta_{0}^{r}=0$ and $\zeta_{r}^{r}=1$. We refer to
$\zeta^{r}$ as the standard gamma bridge of length $r$ associated
with $\gamma$. We note that $\zeta^{r}$ is also called the Dirichlet process with parameter $r$.
\end{definition}
	We note that the process $\zeta^{r}$ is really a function of the variables $(r,t,\omega)$ and for technical reasons, it is convenient to have certain joint measurability properties.
			\begin{lemma}
				The map $(r,t,\omega)\longmapsto \zeta_t^{r}(\omega)$ of $\big((0,+\infty)\times \mathbb{R}_{+} \times \Omega ,\mathcal{B}\big((0,+\infty)\big)\otimes \mathcal{B}(\mathbb{R}_{+})\otimes\mathcal{F}\big)$ into $(\mathbb{R}_{+},\mathcal{B}(\mathbb{R}_{+}))$ is measurable. In particular, the t-section of
				$(r,t,\omega)\longmapsto \zeta_t^{r}(\omega)$: $(r,\omega)\longmapsto \zeta_t^{r}(\omega)$ is measurable with respect to the $\sigma$-algebra
				$\mathcal{B}\big((0,+\infty)\big)\otimes\mathcal{F}$, for all $t \geq 0$.
			\end{lemma}
			\begin{proof}
				Since the map $(r,t)\longmapsto t\wedge r$ is Lipschitz continuous
and $t\mapsto\gamma_{t}$ is càdlàg for all almost $\omega\in\Omega$,
then the map $(r,t,\omega)\longmapsto\zeta_{t}^{r}(\omega)$ can be
obtained as the pointwise limit of sequences of measurable functions.
So, it is sufficient to use standard results on the passage to the
limit of sequences of measurable functions.
			\end{proof}
			As a consequence we have the following corollary.
			\begin{corollary}\label{cormesurable}
				The map $(r,\omega)\longmapsto \zeta_t^{r}(\omega)$ of $\big((0,+\infty)\times \Omega ,\mathcal{B}\big((0,+\infty)\big)\otimes \mathcal{F}\big)$
				into $(\mathcal{D}, \mathcal{B}\left(\mathcal{D}\right))$ is measurable.
			\end{corollary}
A number of properties of the gamma bridge $\zeta^{r}$ sample paths can be easily deduced from the corresponding properties of the gamma sample paths. Hence, we have 

\begin{proposition}
The gamma bridge $\zeta_t^{r},\, t\geq 0$, has the following properties: 
\begin{description}
\item[(i)] $\zeta^{r}$ is a purely jump process and its jumping times
are countable and dense in $[0,r]$ a.s.;
\item[(ii)] the map $t\mapsto \zeta_t^{r}$ is strictly increasing and not continuous
anywhere in $[0,r]$ a.s.; 
\item[(iii)] $\zeta^{r}$ has sample paths of finite variation in $[0,+\infty)$ a.s.;
\item[(iv)]  $\zeta^{r}$ has the following representation:
\begin{equation}
\zeta^{r}_{t}=\underset{k\geq1}{\sum}\,\,\dfrac{J_{k}^{r}}{\underset{{j\geq1}}{\sum}J_{j}^{r}} \,\mathbb{I}_{\{ U_{k}^{r}\leq t\}},\quad t\geq 0.\label{bridgegammarep}
\end{equation}
\end{description}
\end{proposition}
We now turn to distributional properties of the gamma bridge.	
		\begin{proposition}
			\begin{enumerate}
				
				\item[(i)] For all $0<t<r$, the random variable $\zeta_t^{r}$ has a beta distribution $\beta (t,r-t)$ i.e. its density function is given by
				\begin{equation}\label{equationbridgedensity}
				\varphi_{\zeta_{t}^{r}}(x)=\dfrac{\Gamma(r)}{\Gamma(t)\Gamma(r-t)}\,x^{t-1}(1-x)^{r-t-1}\,\mathbb{I}_{\left(0,1\right)}(x).
				\end{equation}
				\item[(ii)] For any $0=t_{0}<t_{1}<\ldots<t_{n}=r$, the vector $\left(\zeta_{t_{1}}^{r}-\zeta_{t_{0}}^{r},\ldots,\zeta_{t_{n}}^{r}-\zeta_{t_{n-1}}^{r}\right)$
is independent from $\gamma_{r}$, with density 
\[
\dfrac{\Gamma(r)}{\overset{n}{\underset{i=1}{\prod}}\;\Gamma(t_{i}-t_{i-1})}\,\overset{n}{\underset{i=1}{\prod}}\;x_{i}^{t_{i}-t_{i-1}-1}
\]
with respect to the Lebesgue measure $dx_{1}\ldots dx_{n-1}$ (or,
as well, $dx_{2}\ldots dx_{n}$) on the simplex 
\[
\left\{ \left(x_{1},\ldots,x_{n}\right):\,x_{i}\geq0,\,x_{1}+\ldots+x_{n}=1\right\} .
\]
					\item[(iii)] For all $t<u<r$ and $x\in (0,1)$, the regular conditional law of
		$\zeta_u^{r}$ given $\zeta_t^{r}=x$ is given by:  
					
					\begin{eqnarray}
					\mathbb{P}(\zeta_{u}^{r}\in dy|\zeta_{t}^{r}=x)&=&\dfrac{\Gamma(r-t)}{\Gamma(u-t)\Gamma(r-u)}\,\dfrac{(y-x)^{u-t-1}(1-y)^{r-u-1}}{(1-x)^{r-t-1}}\,\mathbb{I}_{\{x<y<1\}}\,dy,\label{equationtransitionlaw}
					\end{eqnarray}			
				\end{enumerate}
			\end{proposition}

In the same spirit as in the Proposition \ref{propstrongpropertiesofgamma} we have

			\begin{proposition}\label{lemmaN^rmartingale}
			 \begin{description}
			 			 \item[(i)] $\zeta^{r}$ is a Markov process with respect to its natural filtration.	
			  \item[(ii)] $\zeta^{r}$ satisfies the following equation 
					\begin{equation}
				\zeta_{t}^{r}=N_{t}^{r}+\int_{0}^{t}\dfrac{1-\zeta_{s}^{r}}{r-s}ds, \quad t\in[0, r],\label{equationN^rmartingale}
					\end{equation}
					where $(N^r_t, t\in[0, r])$ is a $\mathbb{F}^{\zeta^{r}}$-martingale.
					\end{description}

			\end{proposition}
			\begin{proof}
(i)	 From Theorem 1.3 in Blumenthal and Getoor \cite{BG} it suffices to prove that for every bounded measurable function $g$ we have:
					\begin{equation}
					\mathbb{E}[g(\zeta_u^{r})|\zeta^{r}_{t_1},\ldots,\zeta^{r}_{t_n}]=\mathbb{E}[g(\zeta_u^{r})|\zeta^{r}_{t_n}],\label{equationmarkovzeta^r}
					\end{equation}
					for all $0\leq t_1< \ldots < t_n< u \leq r$ and for all $n\geq 1$.\\
					Using Proposition \ref{propstrongpropertiesofgamma} (i) we have
					\begin{align*}
					\mathbb{E}[g(\zeta_u^{r})|\zeta^{r}_{t_1},\ldots,\zeta^{r}_{t_n}]&=\mathbb{E}\bigg[g(\dfrac{\gamma_u}{\gamma_r})|\dfrac{\gamma_{t_1}}{\gamma_r},\ldots,\dfrac{\gamma_{t_n}}{\gamma_r}\bigg]\\
					&=\mathbb{E}\bigg[g(\dfrac{\gamma_u}{\gamma_r})|\dfrac{\gamma_{t_1}}{\gamma_{t_2}},\dfrac{\gamma_{t_2}}{\gamma_{t_3}},\ldots,\dfrac{\gamma_{t_{n-1}}}{\gamma_{t_n}},\dfrac{\gamma_{t_n}}{\gamma_r}\bigg]\\
					&=\mathbb{E}\bigg[g(\dfrac{\gamma_u}{\gamma_r})|\dfrac{\gamma_{t_n}}{\gamma_r}\bigg]\\
					&=\mathbb{E}\bigg[g(\zeta_u^{r})|\zeta^{r}_{t_n}\bigg].
					\end{align*}
					Hence the formula \eqref{equationmarkovzeta^r} is proved, then $\zeta^{r}$ is a Markov process with respect to its natural filtration.
					
		(ii)		We have from Proposition \ref{propstrongpropertiesofgamma} (ii) that $$\gamma_t=M_t^r+\int_0^t\dfrac{\gamma_r -\gamma_s}{r-s} ds, \quad t\in[0, r],$$ where $M^r$ is a martingale with respect to the filtration $\mathcal{G}_t^{(r)}=\sigma(\gamma_s, s\in [0, t]\cup \{r\})$. Then it is easy to see that
	\begin{equation}
	\zeta_{t}^{r}=N_t^r+\int_{0}^{t}\dfrac{1-\zeta_{s}^{r}}{r-s}ds,\quad t\in[0, r],\label{xirep}
	\end{equation}
		where $N_t^r=\dfrac{M_t^r}{\gamma_r}, t\in[0, r]$. Firstly, notice that, $\mathcal{F}_t^{\zeta^{r}}\subset\mathcal{G}_t^{(r)}$ and $\gamma_r$ is $\mathcal{G}_t^{(r)}$-measurable for all $t\leq r$.  Moreover, equation \eqref{xirep} yields that is  the process $ N^r$ is $\mathbb{F}^{\zeta^{r}}$-adapted. In view of these considerations, as well as the fact that $M_t^r$ is a $\mathcal{G}_t^{(r)}$-martingale we obtain  
					\[
\begin{array}{lll}
\mathbb{E}\bigg[N_{t}^{r}\vert\mathcal{F}_{s}^{\zeta^{r}}\bigg] & = & \mathbb{E}\bigg[\dfrac{M_{t}^{r}}{\gamma_{r}}\vert\mathcal{F}_{s}^{\zeta^{r}}\bigg]=\mathbb{E}\left[\mathbb{E}\bigg[\dfrac{M_{t}^{r}}{\gamma_{r}}\vert\mathcal{G}_{s}^{(r)}\bigg]\vert\mathcal{F}_{s}^{\zeta^{r}}\right]\\
\\
 & = & \mathbb{E}\left[\dfrac{M_{s}^{r}}{\gamma_{r}}\vert\mathcal{F}_{s}^{\zeta^{r}}\right]=\mathbb{E}\left[N_{s}^{r}\vert\mathcal{F}_{s}^{\zeta^{r}}\right]=N_{s}^{r},
\end{array}
\]
for $0\leq s\leq t\leq r$. It follows that $(N^r_t, t\in[0, r])$ is a $\mathbb{F}^{\zeta^{r}}$-martingale. Hence the equation \eqref{equationN^rmartingale} is satisfied.
		\end{proof}
	\begin{remark}\label{stoppedmart}
	We can rewrite \eqref{gammarepmart} in the form 
	\begin{equation}
	\gamma_{t\wedge r}=M_{t\wedge r}^r+\int_0^{t\wedge r}\dfrac{\gamma_r -\gamma_s}{r-s} ds, \quad t\geq 0.
	\end{equation} Then we obtain 
	\begin{equation}
	\zeta_{t}^{r}=\dfrac{\gamma_{t\wedge r}}{\gamma_{ r}}=\dfrac{M_{t\wedge r}^r}{\gamma_r}+\int_{0}^{t\wedge r}\dfrac{1-\zeta_{s}^{r}}{r-s}ds,\quad t \geq 0.
	\end{equation}
	 For every $t\geq 0$, we set $\widehat{N}_t^r=\dfrac{M_{t\wedge r}^r}{\gamma_r}$. We have thus 
	 
	 \begin{equation}\label{bridgemartrep}
	\zeta_{t}^{r}=\widehat{N}_t^r+\int_{0}^{t}\dfrac{1-\zeta_{s}^{r}}{r-s}\,\mathbb{I}_{\{s<r\}} \,ds,\quad t \geq 0.
	\end{equation}
	 It follows from the above proposition that $(\widehat{N}_t^r, t\geq 0)$ is a $\mathbb{F}^{\zeta^{r}}$-martingale stopped at $r$.
	\end{remark}	

\begin{center}
	\section{Gamma bridges with random length}\label{sectionstoppingtimeproperty}
\end{center}
In this section we define and study a process $(\zeta_t, t\geq 0)$ which generalizes the gamma bridge in the sense that the time $r$ at which the bridge is pinned is substituted by an independent random time $\tau$. We call it \textit{gamma bridge with random length}. We prove that the random time $\tau$ is a stopping time with respect to the completed filtration $\mathbb{F}^{\zeta,c}$ and we give the regular conditional distribution of $\tau$ and $(\tau,\zeta_.)$ given $\zeta_.$. Moreover, we prove that the gamma bridge with random length $\zeta$ is an inhomogeneous Markov process with respect to its completed natural filtration $\mathbb{F}^{\zeta,c}$ as well as with respect to $\mathbb{F}^{\zeta,c}_+$. The last property allows us to deduce an interesting consequence that is the filtration $\mathbb{F}^{\zeta,c}$ satisfies the usual conditions of completeness and right-continuity. Finally we give the semimartingale decomposition of $\zeta$ with respect to $\mathbb{F}^{\zeta,c}_+$.\\
	 
	 Now we give precise definition of the process $(\zeta_t, t\geq 0)$. Due to Corollary \ref{cormesurable} we could substitute $r$ by a random time $\tau$ in \eqref{rgammabridge}. Thus we obtain
	\begin{definition}
		Let $\tau: (\Omega,\mathcal{F},\mathbb{P}) \longmapsto (0,+\infty)$ be a strictly positive random time, with distribution function $F(t) := \mathbb{P}(\tau \leq t)$, $t \geq 0$.
		The map $\zeta :\Omega,\mathcal{F})\longrightarrow (\mathcal{D}, \mathcal{B}\left(\mathcal{D}\right)$ is defined by 
		$$\zeta_{t}(\omega):=\zeta_{t}^{r}(\omega)\vert_{r=\tau(\omega)}~~, (t,\omega) \in \mathbb{R}_{+} \times \Omega .$$
		Then $\zeta$ takes the form	
		\begin{equation}
		\zeta_{t}:=\dfrac{\gamma_{t\wedge \tau}}{\gamma_{\tau}},~~ t\geq 0. \label{defzeta}
		\end{equation}
	\end{definition}
	Since $\zeta$ is obtained by composition of two maps $(r,t,\omega)\longmapsto \zeta^{r}_t(\omega)$ and $(t,\omega)\longmapsto(\tau(\omega), t, \omega)$, it's not hard to verify that the map $\zeta :\Omega,\mathcal{F})\longrightarrow (\mathcal{D}, \mathcal{B}\left(\mathcal{D}\right)$ is measurable. The process $\zeta$ will be called \textit{gamma bridge of random length $\tau$}.
	
	As mentioned above, we work under the following standing assumption:
	\begin{hy}\label{hyindependent}
		The random time $\tau$ and the gamma process $\gamma$ are independent.
	\end{hy} 

Using the fact that the process $\zeta$ is obtained by the substitution of $r$ in $\zeta^{r}$ by the random time $\tau$ allows us to derive a lot of information about its path properties. Hence, we have 

\begin{proposition}
The gamma bridge $\zeta_t,\, t\geq 0$, has the following properties: 
\begin{description}
\item[(i)] $\zeta$ is a purely jump process and its jumping times
are countable and dense in $[0,\tau]$ a.s.;
\item[(ii)] the map $t\mapsto \zeta_t$ is increasing and not continuous
anywhere on $[0,\tau]$  a.s.; 
\item[(iii)] $\zeta$ has sample paths of finite variation a.s.
\item[(iv)]  $\zeta$ has the following representation:
\[
\zeta_{t}=\underset{k\geq1}{\sum}\,\,\dfrac{J_{k}^{\tau}}{\underset{{j\geq1}}{\sum}J_{j}^{\tau}}\,\mathbb{I}_{\{U_{k}^{\tau}\leq t\}},\quad t\geq0,
\]
where the jump times $\left(U_{1}^{\tau},U_{2}^{\tau},\ldots\right)$ constitute
a sequence of r.v.\textquoteright s identically distributed with the law given by
\[
\begin{array}{lll}
\mathbb{P}\left[U_{k}^{\tau}\leq t\right] & = & \mathbb{P}\left[\tau\leq t\right]+\dint_{(t,+\infty)}\,\mathbb{P}\left[U_{k}^{r}\leq t\right]\,\mathbb{P}_{\tau}(dr)\\
\\
 & = & \mathbb{P}\left[\tau\leq t\right]+t\,\mathbb{E}\left[\dfrac{1}{\tau}\,\mathbb{I}_{(\tau> t)}\right], \, t \geq 0, k\geq 1.
\end{array}
\]\end{description}
\end{proposition}

\subsection{Stopping time property of $\tau$}


	The aim of this subsection is to prove that the random time $\tau$ is a stopping time
	with respect to $\mathbb{F}^{\zeta,c}$.

	\begin{proposition}
        For all $t>0$, we have $\mathbb{P}\left(\{\zeta_t = 1\} \bigtriangleup \{\tau \leq t\}\right)=0$. 
		Then $\tau $ is a stopping time with respect to $\mathbb{F}^{\zeta,c}$ and consequently it is a stopping time with respect to $\mathbb{F}^{\zeta,c}_{+}$.
	\end{proposition}
	\begin{proof}
		First we have from the definition of $\zeta$ that  $\zeta_t=1$ for $\tau \leq t$. Then $\{\tau\leq t\}\subseteq \{\zeta_t=1\}$. On the other hand, using the formula of total probability we obtain
		\begin{align*}
		\mathbb{P}(\zeta_{t}=1,t<\tau)&=\dint_{(t,+\infty)} \mathbb{P}(\zeta_{t}=1 | \tau=r) \mathbb{P}_{\tau}(dr)
		\\  
		&=\dint_{(t,+\infty)} \mathbb{P}(\zeta_{t}^{r}=1) \mathbb{P}_{\tau}(dr) \\
		&=0.
		\end{align*}
	    The latter equality uses the fact that $\zeta_t^{r}$ is a random variable has a beta distribution for $0 < t < r$. 
		Thus $\mathbb{P}\left(\{\zeta_t=1\} \bigtriangleup \{\tau \leq t\}\right)=0$.
		It follows that the event $\{\tau \leq t\}$ belongs to  $\mathcal{F}_t^{\zeta} \vee \mathcal{N}_P,$ for all $t \geq 0$. Hence $\tau$ is a stopping time with respect to $\mathbb{F}^{\zeta,c}$ and consequently it is also a stopping time with respect to $\mathbb{F}^{\zeta,c}_{+}$.
	\end{proof}
	In order to determine the conditional law of the random time $\tau$ given $\zeta_{t}$ we will use the following
	\begin{proposition}\label{propbayesestimatejusquatn}
		Let $t>0$ such that $F(t)>0$.  Let $g:\mathbb{R}_{+}\longrightarrow \mathbb{R} $ be a Borel function satisfying $\mathbb{E}[|g(\tau)|]<+\infty$. Then, $\mathbb{P}$-a.s., we have	
		\begin{equation}\label{taucondxi}
		\mathbb{E}[g(\tau)\vert\zeta_{t}]= \dint_{(0,t]}\frac{g(r)}{F(t)}\mathbb{P}_{\tau}(dr)\;\mathbb{I}_{\{\zeta_{t}=1\}}
		+\dint_{(t,+\infty)}g(r)\phi_{\zeta_{t}^{r}}(\zeta_{t})\mathbb{P}_{\tau}(dr)\; \mathbb{I}_{\{0<\zeta_{t}<1\}},
		\end{equation}
		where the function $\phi_{\zeta_{t}^{r}}$ is defined on $\mathbb{R}$ by:
		\begin{eqnarray}
		\phi_{\zeta_{t}^{r}}(x)&=&\dfrac{\varphi_{\zeta_{t}^{r}}(x)}{\dint_{(t,+\infty)}\varphi_{\zeta_{t}^{s}}(x)\mathbb{P}_{\tau}(ds)}\nonumber\\
		&=&\dfrac{\left(1-x\right)^{r}\dfrac{\Gamma(r)}{\Gamma(r-t)}}{\dint_{(t_,+\infty)}\left(1-x\right)^{s}\dfrac{\Gamma(s)}{\Gamma(s-t)}\mathbb{P}_{\tau}(ds)}\,\mathbb{I}_{(0,1)}(x),\quad  x\in \mathbb{R},~~r\in (t,+\infty). \label{equationphi}
		\end{eqnarray}
	\end{proposition}
	\begin{proof}
		Let us consider the measure $\mu$ defined on $\mathcal{B}(\mathbb{R})$ by
		$$ \mu(dx)=\delta_1(dx)+dx,$$ 	
		where $\delta_1(dx)$ and $dx$ are the Dirac measure and the Lebesgue measure on $\mathcal{B}(\mathbb{R})$ respectively. Then for any $B\in \mathcal{B}(\mathbb{R})$ we have   
		$$\mathbb{P}(\zeta_{t}\in B|\tau =r)=\mathbb{P}(\zeta_{t}^{r}\in B)=\dint_{B}q_{t}(r,x)\mu(dx),$$
			where the function $q_{t}$ is a nonnegative and measurable in the two variables jointly given by $$q_{t}(r,x)=\mathbb{I}_{\{x=1 \}}\mathbb{I}_{\{r\leq t \}}+\varphi_{\zeta_{t}^{r}}(x)\mathbb{I}_{\{0< x< 1 \}}\mathbb{I}_{\{t<r \}}.$$ It follows from Bayes formula (see  \cite{S}  p. 272)
		that $\mathbb{P}$-a.s.:
	
		\begin{align*}\mathbb{E}[g(\tau)|\zeta_{t}] & =\frac{\dint_{(0,+\infty)}g(r)q_{t}(r,\zeta_{t})\mathbb{P}_{\tau}(dr)}{\dint_{(0,+\infty)}q_{t}(r,\zeta_{t})\mathbb{P}_{\tau}(dr)}\\
 & =\frac{\dint_{(0,t]}g(r)\mathbb{P}_{\tau}(dr)\mathbb{I}_{\{\zeta_{t}=1\}}+\dint_{(t,+\infty)}g(r)\varphi_{\zeta_{t}^{r}}(\zeta_{t})\mathbb{P}_{\tau}(dr)\mathbb{I}_{\{0<\zeta_{t}<1\}}}{F(t)\mathbb{I}_{\{\zeta_{t}=1\}}+\dint_{(t,+\infty)}\varphi_{\zeta_{t}^{r}}(\zeta_{t})\mathbb{P}_{\tau}(dr)\mathbb{I}_{\{0<\zeta_{t}<1\}}}\\
 & =\dint_{(0,t]}\frac{g(r)}{F(t)}\mathbb{P}_{\tau}(dr)\mathbb{I}_{\{\zeta_{t}=1\}}+\dint_{(t,+\infty)}g(r)\phi_{\zeta_{t}^{r}}(\zeta_{t})\mathbb{P}_{\tau}(dr)\mathbb{I}_{\{0<\zeta_{t}<1\}}.
\end{align*}

			\end{proof}	
			\begin{corollary}
				The conditional law of the random time $\tau$ given $\zeta_{t}$ is given by 
				\begin{equation}\label{taucondlawmulti}
				\mathbb{P}_{\tau\vert\zeta_{t}=x}(x,dr)= \dfrac{1}{F(t)}\,\mathbb{I}_{\{x=1\}}\,\mathbb{I}_{(0,t]}(r)\,\mathbb{P}_{\tau}(dr) +\phi_{\zeta_{t}^{r}}(x)\,\mathbb{I}_{\{0<x< 1\}}\,\mathbb{I}_{(t,+\infty)}(r)\,\mathbb{P}_{\tau}(dr)
				\end{equation}
			\end{corollary}
		The previous proposition can be expanded as follows
				\begin{proposition}\label{ncordbayesest}
					Let $u>t>0$ such that $F(t)>0$.  Let $\mathfrak{g}$ be a bounded measurable function defined on $(0,+\infty)\times \mathbb{R}$. Then, $\mathbb{P}$-a.s., we have		
					
						\begin{equation}\label{tauximultcond}
						\mathbb{E}[\mathfrak{g}(\tau,\zeta_{t})\vert\zeta_{t}]=\dint_{(0,t]}\,\dfrac{\mathfrak{g}(r,1)}{F(t)}\,\mathbb{P}_{\tau}(dr)\,\mathbb{I}_{\{\zeta_{t}=1\}}+\dint_{(t,+\infty)}\,\mathfrak{g}(r,\zeta_{t})\phi_{\zeta_{t}^{r}}(\zeta_{t})\,\mathbb{P}_{\tau}(dr)\,\mathbb{I}_{\{0<\zeta_{t}<1\}},
						\end{equation}
	and						
						\begin{align} \mathbb{E}[\mathfrak{g}(\tau,\zeta_{u})|\zeta_{t}]= & \dint_{(0,t]}\frac{\mathfrak{g}(r,1)}{F(t)}\mathbb{P}_{\tau}(dr)\mathbb{I}_{\{\zeta_{t}=1\}}+\dint_{(t,u]}\mathfrak{g}(r,1)\phi_{\zeta_{t}^{r}}(\zeta_{t})\mathbb{P}_{\tau}(dr)\mathbb{I}_{\{0<\zeta_{t}<1\}}\nonumber\\
\nonumber\\
+ & \dint_{(u,+\infty)}\,\mathfrak{G}_{t,u}(r,\zeta_{t})\phi_{\zeta_{t}^{r}}(\zeta_{t})\,\mathbb{P}_{\tau}(dr)\,\mathbb{I}_{\{0<\zeta_{t}<1\}}.\label{equationbeyesextensionxi_{u},t_{n}<u}
\end{align}
Here the function $\mathfrak{G}_{t,u}(r,\cdot)$ is defined by 
					\begin{eqnarray}
\mathfrak{G}_{t,u}(r,x) & := & \mathbb{E}[\mathfrak{g}(r,\zeta_{u}^{r})|\zeta_{t}^{r}=x]\nonumber\\
 & = & \dint_{\mathbb{R}}\,\mathfrak{g}(r,y)\,\mathbb{P}_{\zeta_{u}^{r}|\zeta_{t}^{r}=x}(dy).\label{Gturx}
\end{eqnarray}
				\end{proposition}
				\begin{proof}
					First of all, it is easy to see that \eqref{tauximultcond} is an immediate consequence of Proposition \ref{propbayesestimatejusquatn}.
					Now	to show \eqref{equationbeyesextensionxi_{u},t_{n}<u}  we begin with by splitting $\mathbb{E}[\mathfrak{g}(\tau,\zeta_{u})\vert\zeta_{t}]$ as follows
						\[
						\mathbb{E}[\mathfrak{g}(\tau,\zeta_{u})\vert\zeta_{t}]=\mathbb{E}[\mathfrak{g}(\tau,1)\mathbb{I}_{\{\tau\leq t\}}\vert\zeta_{t}]+\mathbb{E}[\mathfrak{g}(\tau,1)\mathbb{I}_{\{t<\tau\leq u\}}\vert\zeta_{t}]+\mathbb{E}[\mathfrak{g}(\tau,\zeta_{u})\mathbb{I}_{\{u<\tau\}}\vert\zeta_{t}]
						\]
						We obtain from Proposition \ref{propbayesestimatejusquatn} that
						\[
						\mathbb{E}[\mathfrak{g}(\tau,1)\mathbb{I}_{\{\tau\leq t\}}\vert\zeta_{t}]=\dint_{(0,t]}\frac{\mathfrak{g}(r,1)}{F(t)}\mathbb{P}_{\tau}(dr)\mathbb{I}_{\{\zeta_{t}=1\}}
						\]
						and
						\[
			\mathbb{E}[\mathfrak{g}(\tau,1)\mathbb{I}_{\{t<\tau\leq u\}}\vert\zeta_{t}]=\dint_{(t,u]}\,\mathfrak{g}(r,1)\phi_{\zeta_{t}^{r}}(\zeta_{t})\mathbb{P}_{\tau}(dr) \,\mathbb{I}_{\{0<\zeta_{t}<1\}}.
					\]
					Next we prove that 
					\begin{equation}
					\mathbb{E}[\mathfrak{g}(\tau,\zeta_{u})\mathbb{I}_{\{u<\tau\}}|\zeta_{t}]=\dint_{(u,+\infty)}\,\mathfrak{G}_{t,u}(r,\zeta_{t})\phi_{\zeta_{t}^{r}}(\zeta_{t})\,\mathbb{P}_{\tau}(dr)\,\mathbb{I}_{\{0<\zeta_{t}<1\}}.	\label{condeqtauxi}
					\end{equation}
					Indeed, for a bounded Borel function $h$ we have  
					\begin{align*}\mathbb{E}[\mathfrak{g}(\tau,\zeta_{u})\mathbb{I}_{\left\{ u<\tau\right\} }h(\zeta_{t})] & =\dint_{(u,+\infty)}E[\mathfrak{g}(r,\zeta_{u}^{r})h(\zeta_{t}^{r})]\mathbb{P}_{\tau}(dr)\\
 & =\dint_{(u,+\infty)}\mathbb{E}[\mathbb{E}[\mathfrak{g}(r,\zeta_{u}^{r})h(\zeta_{t}^{r})|\zeta_{t}^{r}]]\mathbb{P}_{\tau}(dr)\\
 & =\dint_{(u,+\infty)}\mathbb{E}[\mathbb{E}[\mathfrak{g}(r,\zeta_{u}^{r})|\zeta_{t}^{r}]h(\zeta_{t}^{r})]\mathbb{P}_{\tau}(dr).
\end{align*}
					Using \eqref{Gturx}, for $t<u<r$, we get 
					\[
					\begin{array}{lll}
\mathbb{E}[\mathfrak{g}(\tau,\zeta_{u})\mathbb{I}_{\left\{ u<\tau\right\} }h(\zeta_{t})] & = & \dint_{(u,+\infty)}\mathbb{E}[\mathfrak{G}_{t,u}(r,\zeta_{t}^{r})h(\zeta_{t}^{r})]\mathbb{P}_{\tau}(dr)\\
\\
 & = & \mathbb{E}[\mathfrak{G}_{t,u}(\tau,\zeta_{t})\mathbb{I}_{\{u<\tau\}}h(\zeta_{t})].
\end{array}
					\]
					It follows from \eqref{tauximultcond}, that $\mathbb{P}$-a.s.
					
					\begin{align}\mathbb{E}[\mathfrak{G}_{t,u}(\tau,\zeta_{t})\mathbb{I}_{\{u<\tau\}}\vert\zeta_{t}]= & \dint_{(u,+\infty)}\,\mathfrak{G}_{t,u}(r,\zeta_{t})\,\phi_{\zeta_{t}^{r}}(\zeta_{t})\,\mathbb{P}_{\tau}(dr)\,\mathbb{I}_{\{0<\zeta_{t}<1\}}.\end{align}
					This induces that
					\[
					\begin{array}{l}
\mathbb{E}[\mathfrak{g}(\tau,\zeta_{u})\mathbb{I}_{\{u<\tau\}}h(\zeta_{t})]=\mathbb{E}\left[\dint_{(u,+\infty)}\,\mathfrak{G}_{t,u}(r,\zeta_{t})\phi_{\zeta_{t}^{r}}(\zeta_{t})\,\mathbb{P}_{\tau}(dr)\,\mathbb{I}_{\{0<\zeta_{t}<1\}}\,h(\zeta_{t})\right].\end{array}
\]
					Hence the formula \eqref{condeqtauxi} is proved and then the proof of the proposition is completed.
			
			\end{proof}

	\subsection{Markov property of $\zeta$ and Bayes estimate of $\tau$}
	In this part we prove that the gamma bridge with random length $\zeta$ is an inhomogeneous Markov process with respect to its completed natural filtration $\mathbb{F}^{\zeta,c}$.
	\begin{theorem}
		The process $(\zeta_t, t\geq 0)$ is an $\mathbb{F}^{\zeta}$-Markov process. That is, for any $t\geq 0$, we have
		\begin{equation}
		\mathbb{E}[f(\zeta_{t+h})\vert\mathcal{F}_t^{\zeta}]=\mathbb{E}[f(\zeta_{t+h})\vert\zeta_t], \mathbb{P}-a.s.,
		\end{equation}
		for all $t,h\geq 0$ and for every bounded measurable function $f$.
	\end{theorem}
	\begin{proof}
 	First, we would like to mention that since $\zeta_0=0$ almost surely then it is easy to see that
 	$$ \mathbb{E}[f(\zeta_{t+h})\vert\mathcal{F}_0^{\zeta}]=\mathbb{E}[f(\zeta_{t+h})\vert\zeta_0]. $$
 	Let us assume $t>0$. As $\mathbb{I}_{\{\zeta_t=0\}}=\mathbb{I}_{\{\tau\leq t\}}$ $\mathbb{P}$-a.s  we rewrite $\mathbb{E}[f(\zeta_{t+h})\vert\mathcal{F}_t^{\zeta}]$ as follows
 	\begin{align*}\mathbb{E}[f(\zeta_{t+h})\vert\mathcal{F}_{t}^{\zeta}] & =\mathbb{E}[f(\zeta_{t+h})\vert\mathcal{F}_{t}^{\zeta}]\mathbb{I}_{\{\tau\leq t\}}+\mathbb{E}[f(\zeta_{t+h})\vert\mathcal{F}_{t}^{\zeta}]\mathbb{I}_{\{t<\tau\}}\\
 	\\
 & =f(1)\mathbb{I}_{\{\zeta_{t}=1\}}+\mathbb{E}[f(\zeta_{t+h})\vert\mathcal{F}_{t}^{\zeta}]\mathbb{I}_{\{t<\tau\}}.
\end{align*}
So it remains to show that
		\[
		\mathbb{E}[f(\zeta_{t+h})\mathbb{I}_{\{t<\tau\}}\vert\mathcal{F}_{t}^{\zeta}]=\mathbb{E}[f(\zeta_{t+h})\mathbb{I}_{\{t<\tau\}}\vert\zeta_{t}],\,\,\mathbb{P}-a.s.
		\]
		To do this it is enough to verify that 
		\begin{equation}
		\int_{A\cap\{t<\tau\}}f(\zeta_{t+h})d\mathbb{P}=\int_{A\cap\{t<\tau\}}\mathbb{E}[f(\zeta_{t+h})\vert\zeta_{t}]d\mathbb{P}, \label{equationAcapt<tau}
		\end{equation}
		for all $A\in \mathcal{F}_t^{\zeta}$. We start by remarking that, for $t>0$, $\mathcal{F}^{\zeta}_t$ is generated by 
		\[
		\zeta_{t_n},\alpha_n:=\dfrac{\zeta_{t_{n-1}}}{\zeta_{t_n}},\alpha_{n-1}=\dfrac{\zeta_{t_{n-2}}}{\zeta_{t_{n-1}}},\ldots,\alpha_{2}=\dfrac{\zeta_{t_{1}}}{\zeta_{t_{2}}},\alpha_{1}:=\dfrac{\zeta_{t_{0}}}{\zeta_{t_{1}}},
		\]
		$0<t_{0}<t_{1}<\cdots<t_{n}=t$ for n running through $\mathbb{N}$. Then by the monotone class theorem it is sufficient to prove \eqref{equationAcapt<tau} for sets
		$A$ of the form $A=\{\zeta^z_t\in B,\alpha_1\in B_1,\ldots,\alpha_n\in B_n \}$ with $B, B_1, B_2,\ldots, B_n \in
		\mathcal{B}(\mathbb{R})$, $n \geq 1$.
		Moreover, on the set $\{t<\tau\}$, we have 
		\[
		\beta_k:=\dfrac{\gamma_{t_{k-1}}}{\gamma_{t_{k}}}=\alpha_k,\;k=1,\ldots,n.
		\]	
		Using  Proposition \ref{propstrongpropertiesofgamma} (i), then for $t<r$ the vectors  
		  $(\beta_1,\ldots,\beta_n)$ and $(\zeta_t^{r},\zeta_{t+h}^{r})$ are independent.
		 Now taking into account all the above considerations, we have 
\begin{align*}\dint_{A\cap\{t<\tau\}}f(\zeta_{t+h})d\mathbb{P} & =\mathbb{E}\left[f(\zeta_{t+h})\,\mathbb{I}_{B\times B_{1}\times\ldots\times B_{n}}(\zeta_{t},\alpha_{1},\ldots,\alpha_{n})\,\mathbb{I}_{\{t<\tau\}}\right]\\
\\
 & =\mathbb{E}[f(\zeta_{t+h})\,\mathbb{I}_{B\times B_{1}\times\ldots\times B_{n}}(\zeta_{t},\beta_{1},\ldots,\beta_{n})\,\mathbb{I}_{\{t<\tau\}}]\\
\\
 & =\int_{(t,\infty)}\mathbb{E}\left[f(\zeta_{t+h}^{r})\,\mathbb{I}_{B}(\zeta_{t}^{r})\,\mathbb{I}_{B_{1}\times\ldots\times B_{n}}(\beta_{1},\ldots,\beta_{n})\right]\,\mathbb{P}_{\tau}(dr)\\
\\
 & =\int_{(t,\infty)}\mathbb{E}[f(\zeta_{t+h}^{r})\,\mathbb{I}_{B}(\zeta_{t}^{r})]\mathbb{P}_{\tau}(dr)\,\mathbb{E}[\mathbb{I}_{B_{1}\times\ldots\times B_{n}}(\beta_{1},\ldots,\beta_{n})]\\
\\
 & =\mathbb{E}[f(\zeta_{t+h})\,\mathbb{I}_{B}(\zeta_{t})\,\mathbb{I}_{\{t<\tau\}}]\,\mathbb{E}[\mathbb{I}_{B_{1}\times\ldots\times B_{n}}(\beta_{1},\ldots,\beta_{n})]\\
\\
 & =\mathbb{E}\left[\mathbb{E}[f(\zeta_{t+h})\vert\zeta_{t}]\,\mathbb{I}_{B}(\zeta_{t})\,\mathbb{I}_{\{t<\tau\}}]\,\mathbb{E}[\mathbb{I}_{B_{1}\times\ldots\times B_{n}}(\beta_{1},\ldots,\beta_{n})\right]\\
\\
 & =\mathbb{E}[\mathbb{E}[f(\zeta_{t+h})\vert\zeta_{t}]\,\mathbb{I}_{B}(\zeta_{t})\,\mathbb{I}_{\{t<\tau\}}\,\mathbb{I}_{B_{1}\times\ldots\times B_{n}}(\beta_{1},\ldots,\beta_{n})]\\
\\
 & =\mathbb{E}[\mathbb{E}[f(\zeta_{t+h})\vert\zeta_{t}]\,\mathbb{I}_{B}(\zeta_{t})\,\mathbb{I}_{\{t<\tau\}}\,\mathbb{I}_{B_{1}\times\ldots\times B_{n}}(\alpha_{1},\ldots,\alpha_{n})]\\
\\
 & =\mathbb{E}[\mathbb{E}[f(\zeta_{t+h})\vert\zeta_{t}]\,\mathbb{I}_{B\times B_{1}\times\ldots\times B_{n}}(\zeta_{t},\alpha_{1},\ldots,\alpha_{n})\,\mathbb{I}_{\{t<\tau\}}]\\
\\
 & =\int_{A\cap\{t<\tau\}}\mathbb{E}[f(\zeta_{t+h})\vert\zeta_{t}]\,d\mathbb{P}.
\end{align*}
Hence \eqref{equationAcapt<tau} is proved and this ends the proof.

	\end{proof}

	\begin{corollary}
	The Markov property can be extended to the completed filtration  $\mathbb{F}^{\zeta,c}$.
	\end{corollary}
	The aim of this proposition is to provide, using the Markov property, that the observation of $\zeta_{t}$ would be sufficient to give estimates of the time $\tau$ based on the observation of the information process $\zeta$ up to time $t$. 
	 \begin{proposition}\label{propbayesestimate}
	 	Let $0<t<u$.\\
	 	\begin{enumerate}
	 		\item[(i)] For each bounded measurable function $g$ defined on $(0, \infty)$, we have $\mathbb{P}$-a.s.
	 		\begin{equation}
	 		\mathbb{E}[g(\tau)\vert\mathcal{F}_{t}^{\zeta,c}]=g(\tau\wedge t)\mathbb{I}_{\{\zeta_{t}=1\}}+\dint_{(t,+\infty)}g(r)\phi_{\zeta_{t}^{r}}(\zeta_{t})\,\mathbb{P}_{\tau}(dr)\,\mathbb{I}_{\{0<\zeta_{t}<1\}}.\label{equationtaugivenF}
	 		\end{equation}
	 		\item[(ii)] For each bounded measurable function defined on $(0,+\infty)\times \mathbb{R}$, we have $\mathbb{P}$-a.s.
	 		\begin{equation}
	 		\mathbb{E}[g(\tau,\zeta_{t})\vert \mathcal{F}_{t}^{\zeta,c}]=g(\tau\wedge t,1)\mathbb{I}_{\{\zeta_{t}=1\}}+\dint_{(t,+\infty)}g(r,\zeta_{t})\phi_{\zeta_{t}^{r}}(\zeta_{t})\mathbb{P}_{\tau}(dr)\mathbb{I}_{\{0<\zeta_{t}<1\}}.
\label{equationtauzeta_tgivenF}
	 		\end{equation}

	 		\begin{eqnarray}
\mathbb{E}[g(\tau,\zeta_{u})\vert \mathcal{F}_{t}^{\zeta,c}] & = & g(\tau\wedge t,1)\mathbb{I}_{\{\zeta_{t}=1\}}+\dint_{(t,u]}g(r,1)\phi_{\zeta_{t}^{r}}(\zeta_{t})\mathbb{P}_{\tau}(dr)\mathbb{I}_{\{0<\zeta_{t}<1\}}
\nonumber
\\ \nonumber
\\
 &  & +\dint_{(u,+\infty)}\int_{\mathbb{R}}g(r,y)\mathbb{P}_{\zeta_{u}^{r}\vert \zeta_{t}^{r}=x}(dy)\phi_{\zeta_{t}^{r}}(\zeta_{t})\mathbb{P}_{\tau}(dr)\mathbb{I}_{\{0<\zeta_{t}<1\}}. \label{equationtauzetaugivenF}
\end{eqnarray}
	 		
	 	\end{enumerate}
	 \end{proposition}
	 \begin{proof} 
(i)\,	 	 	Obviously, we have
	 		$$\mathbb{E}[g(\tau)\vert \mathcal{F}_{t}^{\zeta,c}]=\mathbb{E}[g(\tau\wedge t)\mathbb{I}_{\{\tau\leqslant t\}}\vert \mathcal{F}_{t}^{\zeta,c}]+\mathbb{E}[g(\tau\vee t)\mathbb{I}_{\{ t<\tau \}}\vert \mathcal{F}_{t}^{\zeta,c}].$$
	 		Now since  $g(\tau\wedge t)\mathbb{I}_{\{\tau\leqslant t\}}$ is $ \mathcal{F}^{\zeta,c}_{t}$-measurable then, $\mathbb{P}$-a.s, one has
	 		\[
	 		\begin{array}{lll}
	 		\mathbb{E}[g(\tau\wedge t)\mathbb{I}_{\{\tau\leqslant t\}}\vert\mathcal{F}_{t}^{\zeta,c}] & = & g(\tau\wedge t)\mathbb{I}_{\{\tau\leqslant t\}}\\
	 		\\
	 		& = & g(\tau\wedge t)\mathbb{I}_{\{\zeta_{t}=1\}}.
	 		\end{array}
	 		\]
	 		On the other hand due to the facts that $g(\tau\vee t)\mathbb{I}_{\{t< \tau \}}$ is  $\sigma(\zeta_s, t \leq s \leq +\infty) \vee \mathcal{N}_P$-measurable and $\zeta$ is a Markov process with respect to its completed natural filtration we obtain $\mathbb{P}$-a.s.
	 		$$\mathbb{E}[g(\tau\vee t)\mathbb{I}_{\{ t<\tau \}}\vert \mathcal{F}_{t}^{\zeta,c}]=\mathbb{E}[g(\tau\vee t)\mathbb{I}_{\{ t<\tau \}}|\zeta_{t}],$$
	 		The result is deduced from \eqref{taucondxi}.

(ii)\,	 		 The equation \eqref{equationtauzeta_tgivenF} is an immediate consequence of (i). Concerning the equation \eqref{equationtauzetaugivenF} we use the same method which we used to prove (i).	
	 \end{proof}
	 \begin{remark} 
	 	The process $\zeta$ cannot be an homogeneous $\mathbb{F}^{\zeta^z}$-Markov process. Indeed, Proposition \ref{propbayesestimate} enables us to see that, for $A \in \mathcal{B}(\mathbb{R})$ and $t<u$, we have $\mathbb{P}$-a.s.,
	 		\[
\begin{array}{lll}
\mathbb{P}(\zeta_{u}\in A\vert\mathcal{F}_{t}^{\zeta}) & = & \mathbb{I}_{\{1\in A\}}\,\mathbb{I}_{\{\zeta_{t}=1\}}+\mathbb{I}_{\{1\in A\}}\dint_{(t,u]}\phi_{\zeta_{t}^{r}}(\zeta_{t})\,\mathbb{P}_{\tau}(dr)\,\mathbb{I}_{\{0<\zeta_{t}<1\}}\\
\\
 &  & +\dint_{(u,+\infty)}\int_{A}\mathbb{I}_{\{\zeta_{t}<y<1\}}\,\mathbb{P}_{\zeta_{u}^{r}\vert\zeta_{t}^{r}=x}(dy)\,\phi_{\zeta_{t}^{r}}(\zeta_{t})\,\mathbb{P}_{\tau}(dr)\,\mathbb{I}_{\{0<\zeta_{t}<1\}},
\end{array}
\]
	 		which is clear that it doesn't depend only on $u - t$.	 	
	 \end{remark}
	 \subsection{Markov property with respect to $\mathbb{F}^{\zeta,c}_+$}	\label{sectionmarkovpropertyfiltrationcontinue}
	 
	 We have established, in the previous section, the Markov property of $\zeta$ with respect to its completed natural filtration $\mathbb{F}^{\zeta,c}$. In this section we are interested in the the Markov property  of $\zeta$ with respect to $\mathbb{F}^{\zeta,c}_{+}$. It has an interesting consequence which is none other than the filtration $\mathbb{F}^{\zeta,c}$ satisfies the usual conditions of completeness and right-continuity. However, we need the following condition of  on the integrability of $\tau$.
\begin{hy}\label{hypconv} There exists a sufficiently small $\alpha >0 $ such that
	\begin{equation}
	\mathbb{E}\left(\tau^\alpha\right)<+\infty.\label{integrcond}
	\end{equation}
\end{hy}
The next theorem shows the Markov property of $\xi$ with respect to $\mathbb{F}^{\xi,c}_+$.
	 \begin{theorem}\label{thmmarkovpropertyfiltrationcontinue} 
	 	The  process $\zeta$ is a Markov process with respect to $\mathbb{F}^{\zeta,c}_{+}$.
	 \end{theorem}
	 \begin{proof}
	 	It is sufficient to prove that for any $0\leq t<u$ and any function bounded continuous $g$ we have
	 	\begin{equation}
	 	\mathbb{E}[g(\zeta_{u})|\mathcal{F}^{\zeta,c}_{t+}]=\mathbb{E}[g(\zeta_{u})|\zeta_{t}],~~ \mathbb{P}-a.s.\label{equationmarkovrct>0}
	 	\end{equation}
	 	Let $(t_{n})_{n\in \mathbb{N}}$
	 	be a decreasing sequence of strictly positive real numbers converging to $t$: that is $0 \leq t <...< t_{n+1} < t_{n}<...<t_{1} < u $, $t_{n} \searrow t$ as $n \longrightarrow +\infty $. Since $g$ is bounded and $\mathcal{F}_{t+}^{\zeta,c}=\underset{n}{\cap}\mathcal{F}_{t_{n}}^{\zeta,c}$ then,  $\mathbb{P}$-a.s., we have
	 	\begin{equation}
	 	\mathbb{E}[g(\zeta_{u})|\mathcal{F}^{\zeta,c}_{t+}]=\lim\limits_{n\longmapsto +\infty}\mathbb{E}[g(\zeta_{u})|\mathcal{F}^{\zeta,c}_{t_n}].
	 	\end{equation}
	 	It follows from the Markov property of $\zeta$ with respect to $\mathbb{F}^{\zeta,c}$ that 	
	 	\begin{equation}
	 	\mathbb{E}[g(\zeta_{u})|\mathcal{F}^{\zeta,c}_{t+}]=\lim\limits_{n\longmapsto +\infty}\mathbb{E}[g(\zeta_{u})|{\zeta}_{t_n}],~~\mathbb{P}-a.s.
	 	\end{equation}
	 	It remains to prove that
	 	\begin{equation}
	 	\lim\limits_{n \longrightarrow +\infty} \mathbb{E}[g(\zeta_{u})|\zeta_{t_{n}}]=\mathbb{E}[g(\zeta_{u})|\zeta_{t}],~~\mathbb{P}-a.s.\label{equationcontdemarkovrc}
	 	\end{equation}
	 	The proof is splitted into two parts. In the first one we show the statment \eqref{equationmarkovrct>0}  for $t > 0$, while in the second part we consider the case $t = 0$.
	 	
	 	Let $t>0$. We begin by noticing that from Proposition \ref{propbayesestimate}, $\mathbb{P}$-a.s., we have
	 	
	 	\begin{align}
	 	\mathbb{E}[g(\zeta_u)|\zeta_{t_{n}}]=& g(1) \left( \mathbb{I}_{\{\zeta_{t_{n}}=1\}} 	+ \dint_{(t_{n},u]}\,\phi_{\zeta_{t_{n}}^{r}}(\zeta_{t_{n}})\mathbb{P}_{\tau}(dr)\mathbb{I}_{\{0<\zeta_{t_{n}}<1\}} \right)\nonumber  \\
	 	\nonumber \\	& +\dint_{(u,+\infty)}\,K_{t_{n},u}(r,\zeta_{t_{n}})\phi_{\zeta_{t_{n}}^{r}}(\zeta_{t_{n}})\,\mathbb{P}_{\tau}(dr)\,\mathbb{I}_{\{0<\zeta_{t_{n}}<1\}}
 \nonumber \\ 
	 	\nonumber \\ & = g(1)\left( \mathbb{I}_{\{\tau \leq  t_{n}\}}+\dint_{(t_{n},u]}\phi_{\zeta_{t_{n}}^{r}}(\zeta_{t_{n}})\mathbb{P}_{\tau}(dr)\mathbb{I}_{\{t_{n}<\tau \}} \right) \nonumber \\ 
	 	\nonumber \\ & +\dint_{(u,+\infty)}\,K_{t_{n},u}(r,\zeta_{t_{n}})\phi_{\zeta_{t_{n}}^{r}}(\zeta_{t_{n}})\,\mathbb{P}_{\tau}(dr)\,\mathbb{I}_{\{t_{n}<\tau\}}.
 \nonumber
	 	\end{align}
	 		 Where the function $K_{t,u}(r,x)$ is defined on $\mathbb{R}$ by for $0<t<u<r$
	 		 \begin{eqnarray}
	 		 K_{t,u}(r,x)&:=&\mathbb{E}[g(\zeta_{u}^{r})|\zeta_{t}^{r}=x]\nonumber \\&=&\dint_{\mathbb{R}}g(y)\mathbb{I}_{\{x<y<1\}}\mathbb{P}_{\zeta_{u}^{r}|\zeta_{t}^{r}=x}(dy).
 \label{Hcondlaw}
	 		 \end{eqnarray}	
	 	Since $\lim\limits_{n \longrightarrow +\infty}\, \mathbb{I}_{\{t_{n}<\tau \}}=\mathbb{I}_{\{t<\tau \}} $ then assertion \eqref{equationcontdemarkovrc} will be established if we show, $\mathbb{P}$-a.s on $\left\{ t<\tau\right\} $, that
	 	\begin{equation}
	 	\lim\limits _{n\longrightarrow+\infty}\dint_{(t_{n},u]}\,\phi_{\zeta_{t_{n}}^{r}}(\zeta_{t_{n}})\mathbb{P}_{\tau}(dr)=\dint_{(t,u]}\,\phi_{\zeta_{t}^{r}}(\zeta_{t})\mathbb{P}_{\tau}(dr),\label{equationlimitphitnxi-phitxi}
	 	\end{equation}
	 	and
	 	\begin{equation}
	 	\lim\limits _{n\longrightarrow+\infty}\dint_{(u,+\infty)}\,K_{t_{n},u}(r,\zeta_{t_{n}})\phi_{\zeta_{t_{n}}^{r}}(\zeta_{t_{n}})\,\mathbb{P}_{\tau}(dr)=\dint_{(u,+\infty)}\,K_{t,u}(r,\zeta_{t})\phi_{\zeta_{t}^{r}}(\zeta_{t})\,\mathbb{P}_{\tau}(dr). \label{equationlimitphitnxiG-phitxiG}
	 	\end{equation}
	 	We start by proving	 assertion \eqref{equationlimitphitnxi-phitxi}. The integral on the left-hand side of \eqref{equationlimitphitnxi-phitxi} can be rewritten as
	 	\begin{align*}\dint_{(t_{n},u]}\phi_{\zeta_{t_{n}}^{r}}(\zeta_{t_{n}})\,\mathbb{P}_{\tau}(dr) & =\frac{\dint_{(t_{n},u]}\varphi_{\zeta_{t_{n}}^{r}}(\zeta_{t_{n}})\,\mathbb{P}_{\tau}(dr)}{\dint_{(t_{n},+\infty)}\varphi_{\zeta_{t_{n}}^{s}}(\zeta_{t_{n}})\mathbb{P}_{\tau}(ds)}\,\mathbb{I}_{\left\lbrace 0<\zeta_{t_{n}}<1\right\rbrace}\\
 & =\dfrac{\dint_{(t_{n},u]}\bigg(1-\zeta_{t_{n}}\bigg)^{r}\dfrac{\Gamma(r)}{\Gamma(r-t_{n})}\mathbb{P}_{\tau}(dr)}{\dint_{(t_{n},+\infty)}\bigg(1-\zeta_{t_{n}}\bigg)^{s}\dfrac{\Gamma(s)}{\Gamma(s-t_{n})}\mathbb{P}_{\tau}(ds)}\,\mathbb{I}_{\left\lbrace 0<\zeta_{t_{n}}<1\right\rbrace}.
\end{align*}
	 	First let us remark that the function 
	 	\[
	 	(t,r,x)\longrightarrow (1-x)^{r}\dfrac{\Gamma(r)}{\Gamma(r-t)}
	 	\]
	 	defined on $\ensuremath{\{(t,r)\in (0,+\infty)^2, t< r\}\times (0, 1)}$
	 	is continuous. Using the facts that $\zeta_{t_{n}}$ is decreasing to $\zeta_t$ and $\mathbb{P}\left[\zeta_t=0\right]=0$ then, $\mathbb{P}$-a.s on $\left\{ t<\tau\right\} $,
	 	we have 
	 	\begin{equation}
	 	\underset{n\rightarrow+\infty}{\lim}\,(1-\zeta_{t_{n}})^{r}\dfrac{\Gamma(r)}{\Gamma(r-t_{n})}\mathbb{I}_{\{t_{n}<r\}}\,\mathbb{I}_{\left\lbrace 0<\zeta_{t_{n}}<1\right\rbrace}=(1-\zeta_{t})^{r}\dfrac{\Gamma(r)}{\Gamma(r-t)}\mathbb{I}_{\{t<r\}}\,\mathbb{I}_{\left\lbrace 0<\zeta_{t}<1\right\rbrace}.
\label{phixiconv}
	 	\end{equation}
	 	On the other hand since the function $x\longmapsto\left(1-x\right)^{r}$ is decreasing on $(0,1)$ for all $r>0$ and 
\begin{equation}\label{equivcond}
0\leq\dfrac{\Gamma(r)}{\Gamma(r-t)}=r^{t}\left[1-\dfrac{t(t+1)}{2r}+O\left(\dfrac{1}{r^{2}}\right)\right],
\end{equation} 	
for large enough $r$, see \cite{AB}, p. 257, 6.1.46, then for any compact subset $\mathcal{K}$ of $(0,+\infty)\times (0, 1)$
	 	it yields 
	 	\[
	 	\underset{(t,x)\in\mathcal{K}}{\sup}\,(1-x)^{r}\dfrac{\Gamma(r)}{\Gamma(r-t)}\mathbb{I}_{\{t<r\}}<+\infty.
	 	\]
	 	Hence, $\mathbb{P}$-a.s on $\left\{ t<\tau\right\} $, we have
	 	\begin{equation}
	 	\underset{n\in\mathbb{N}}{\sup}\,\bigg(1-\zeta_{t_{n}}\bigg)^{r}\dfrac{\Gamma(r)}{\Gamma(r-t_{n})}\mathbb{I}_{\{t_{n}<r\}}\,\mathbb{I}_{\left\lbrace 0<\zeta_{t_{n}}<1\right\rbrace}<+\infty.
\label{varphiconvergence}
	 	\end{equation}
	 	We conclude assertion \eqref{equationlimitphitnxi-phitxi} from the Lebesgue dominated convergence theorem.\\
	 	Now let us prove \eqref{equationlimitphitnxiG-phitxiG}. Recall that the
	 	function $K_{t_{n},u}(r,\zeta_{t_{n}})$ is given by
\[	 	\begin{array}{lll}	 
K_{t_{n},u}(r,\zeta_{t_{n}}) & = & \dint_{\mathbb{R}}g(y)\mathbb{I}_{\{x<y<1\}}\mathbb{P}_{\zeta_{u}^{r}|\zeta_{t_{n}}^{r}=x}(dy)\vert_{x=\zeta_{t_{n}}}\\
\\
 & = & \dfrac{\Gamma(r-t_{n})}{\Gamma(u-t_{n})\Gamma(r-u)}\,\dint_{\mathbb{R}}\,g(y)\,\dfrac{(y-\zeta_{t_{n}})^{u-t-1}(1-y)^{r-u-1}}{(1-\zeta_{t_{n}})^{r-t-1}}\,\mathbb{I}_{\{\zeta_{t_{n}}<y<1\}}\,dy.
	\end{array}\]
	 	Since $g$ is bounded we deduce that $K_{t_{n},u}(r,\zeta_{t_{n}})$ is bounded. Moreover we obtain from the weak convergence that  
	 	\[
	 	\underset{n\rightarrow+\infty}{\lim}\, K_{t_{n},u}(r,\zeta_{t_{n}})=K_{t,u}(r,\zeta_{t}),
	 	\] 
	 	$\mathbb{P}$-a.s on $\left\{ t<\tau\right\} $.
	 	Combining the fact that $K_{t_{n},u}(r,\zeta_{t_{n}})$ is bounded, \eqref{phixiconv} and \eqref{varphiconvergence} assertion \eqref{equationlimitphitnxiG-phitxiG} is then derived from the Lebesgue dominated convergence theorem.
	 	
	 	Next, we investigate the second part of the proof, that is the case $t = 0$. It will be carried out in two steps. In the first one we assume that there exists $\varepsilon > 0$ such that 
	 	\begin{equation}\label{tauminor}
	 	\mathbb{P}(\tau > \varepsilon)=1.
	 	\end{equation} 
	 	As in the first part, it is sufficient to verify that
	 	\begin{equation}\label{equationlimitmarkovstep2t=0}
	 	\lim\limits_{n \longrightarrow +\infty} \mathbb{E}[g(\zeta_{u})\vert {\zeta}_{t_{n}}]=\mathbb{E}[g(\zeta_{u})\vert \zeta_{0}],~~\mathbb{P}-a.s.
	 	\end{equation}
	 	Without loss of generality we assume
	 	$t_{n}< \alpha \wedge \varepsilon $ for all $n\in \mathbb{N}$. It is easy to see that under condition \eqref{tauminor}, 	$\mathbb{E}[g(\zeta_{u})\vert\zeta_{t_{n}}]$ takes the form
	 	\[
	 	\mathbb{E}[g(\zeta_{u})\vert\zeta_{t_{n}}]=g(1)\dint_{(\varepsilon,u]}\phi_{\zeta_{t_{n}}^{r}}(\zeta_{t_{n}})\mathbb{P}_{\tau}(dr)+\dint_{(u,+\infty)}\,K_{t_{n},u}(r,\zeta_{t_{n}})\phi_{\zeta_{t_{n}}^{r}}(\zeta_{t_{n}})\,\mathbb{P}_{\tau}(dr).
	 	\]
	 	On the other hand we have
	 	\[
	 	\mathbb{E}[g(\zeta_{u})\vert\zeta_{0}]=\mathbb{E}[g(\zeta_{u})]=g(1)F(u)+\int_{(u,+\infty)}\int_{\mathbb{R}}\,g(y)\varphi_{\zeta_{t}^{r}}(y)\,dy\,\mathbb{P}_{\tau}(dr).
	 	\]
	 	Then in order to show \eqref{equationlimitmarkovstep2t=0} it is sufficient to prove, $\mathbb{P}$-a.s, the following
	 	\begin{equation}
	 	\lim\limits_{n\longrightarrow +\infty}\;\dint_{(\varepsilon,u]}\phi_{\zeta_{t_{n}}^{r}}(\zeta_{t_{n}})\mathbb{P}_{\tau}(dr)=F(u),
	 	\label{conas1}
	 	\end{equation}
	 	and 
	 	\begin{equation}
	 	\lim\limits _{n\longrightarrow+\infty}\,\dint\limits _{(u,+\infty)}\,K_{t_{n},u}(r,\zeta_{t_{n}})\phi_{\zeta_{t_{n}}^{r}}(\zeta_{t_{n}})\,\mathbb{P}_{\tau}(dr)=\int_{(u,+\infty)}\int_{\mathbb{R}}\,g(y)\varphi_{\zeta_{t}^{r}}(y)\,dy\,\mathbb{P}_{\tau}(dr).\label{conas2}
	 	\end{equation}
First, for $r>\varepsilon$, we have
\[
\underset{n\rightarrow+\infty}{\lim}\,(1-\zeta_{t_{n}})^{r}\dfrac{\Gamma(r)}{\Gamma(r-t_{n})}\mathbb{I}_{\{t_{n}<r\}}\mathbb{I}_{\left\{ 0<\zeta_{t_{n}}<1\right\} }=1.
\]
Since the gamma function is increasing on $[2,\infty )$ then, for $r\geq 2+t_{1}$, we obtain 
	 	\begin{equation}
	 	\underset{n\in\mathbb{N}}{\sup}\,\bigg(1-\zeta_{t_{n}}\bigg)^{r}\dfrac{\Gamma(r)}{\Gamma(r-t_{n})}<\dfrac{\Gamma(r)}{\Gamma(r-t_{1})}.
\label{varphiconvergence2}
	 	\end{equation}
It follows from \eqref{integrcond} and \eqref{equivcond} that the fonction $r\longmapsto\dfrac{\Gamma(r)}{\Gamma(r-t_{1})}$ is $\mathbb{P}_{\tau}$-integrable on  $(\varepsilon,+\infty)$. Hence \eqref{conas1} follows from a simple application of the Lebesgue dominated convergence theorem. In the same way as in the first case ($t>0$) we obtain from the weak convergence that 
	 	\[
	 	\underset{n\rightarrow+\infty}{\lim}\,K_{t_{n},u}(r,\zeta_{t_{n}})=\dfrac{\Gamma(r)}{\Gamma(u)\Gamma(r-u)}\,\dint_{\mathbb{R}}\,g(y)\,y^{u-1}(1-y)^{r-u-1}\,\mathbb{I}_{\{0<y<1\}}\,dy,
	 	\] 	
	 	then also \eqref{conas2} follows from a simple application of the Lebesgue dominated convergence theorem.
	 	Finally, we have to consider the general case, that is $\mathbb{P}(\tau >0)=1$. In order to prove the Markov property of $\zeta$ with respect to $\mathbb{F}^{\zeta,c}_{+}$ at $t = 0$ it is sufficient to show that $\mathcal{F}_{0+}^{\zeta,c}$ is $\mathbb{P}$-trivial. This amounts to prove that $\mathcal{F}_{0+}^{\zeta}$ is $\mathbb{P}$-trivial since $\mathcal{F}_{0+}^{\zeta,c}=\mathcal{F}_{0+}^{\zeta} \vee \mathcal{N}_{P}$. To do so, let $\varepsilon > 0$ be fixed and consider the stopping time $\tau_{\varepsilon}=\tau \vee \varepsilon$. We define the process $\zeta_{t}^{\tau_{\varepsilon}}$ by $$\left\lbrace\zeta_{t}^{\tau_{\varepsilon}};\,t\geq 0\right\rbrace:=\left\lbrace\zeta_{t}^{r} \vert_{r=\tau \vee \varepsilon};\,t\geq 0\right\rbrace.$$  
	 	The first remark is that the sets $(\tau_{\varepsilon}>\varepsilon)=(\tau>\varepsilon)$ are equal  and therefore the following equality of processes holds $$\zeta_{\cdot}^{\tau_{\varepsilon}}\mathbb{I}_{(\tau>\varepsilon)}=\zeta_{\cdot}\;\mathbb{I}_{(\tau>\varepsilon)}.$$
	 	Then for each $A\in \mathcal{F}_{0+}^{\zeta}$ there exists $B\in \mathcal{F}_{0+}^{\zeta^{\tau_{\varepsilon}}}$ such that  
	 	$$A\cap(\tau>\varepsilon)=B\cap (\tau>\varepsilon).$$
	 	As $\mathbb{P}(\tau_{\varepsilon}>\varepsilon/2)=1$ then  according to the previous case we have that $\mathcal{F}_{0+}^{\zeta^{\tau_{\varepsilon}}}$ is $\mathbb{P}$-trivial. That is $\mathbb{P}(B)=0$ or $1$. Consequently we obtain 
	 	$$\mathbb{P}(A\cap(\tau>\varepsilon))=0\text{\,\,or\,\,}\mathbb{P}(A\cap(\tau>\varepsilon))=\mathbb{P}(\tau>\varepsilon).$$
	 	Now if $\mathbb{P}(A)>0$, then there exists $\varepsilon > 0$ such that $\mathbb{P}(A\cap \{\tau> \varepsilon\})>0$. Therefore for all $0<\varepsilon'\leq \varepsilon$ we have 
	 	$$\mathbb{P}(A\cap(\tau>\varepsilon'))=\mathbb{P}(\tau>\varepsilon').$$ 
	 	Passing to the limit as $\varepsilon'$ goes to $ 0$ yields $\mathbb{P}(A\cap(\tau>0))=\mathbb{P}(\tau>0)=1$. It follows that $\mathbb{P}(A)=1$, which ends the proof.
	 \end{proof}	
	 \begin{corollary}\label{corFpsatisfiesusualcondition}
	 	The  filtration $\mathbb{F}^{\zeta,c}$ satisfies the usual conditions of rightcontinuity and completeness.
	 \end{corollary}
	 \begin{proof}
	 	See, e.g., [\cite{BG}, Ch. I, Proposition (8.12)]
	 \end{proof}
	 
	 \subsection{Semimartingale Decomposition of $\zeta$}\label{sectionsemimartingale}
Our purpose is to derive the semimartingale property of $\zeta$ with respect to its own filtration $\mathbb{F}^{\zeta,c}$. Firstly,  we obtain from the representation \eqref{bridgemartrep} that 
	 
	\begin{equation}
				\zeta_{t}=\widehat{N}_{t}+\dint_{0}^{t}\, Z_{s}\, ds, \quad t\geq 0,\label{RBrep}
					\end{equation}
					where the processes $\widehat{N}$ and $Z$ are defined as follows:
					$$\widehat{N}_{t}(\omega):=\widehat{N}_t^r(\omega)\vert_{r=\tau(\omega)},$$		and
					$$Z_{t}=\dfrac{1-\zeta_{t}}{\tau-t} \, \mathbb{I}_{\{t< \tau \}},$$
for $(t,\omega) \in \mathbb{R}_{+} \times \Omega$. Now let us consider the filtration 
\begin{equation}
\mathbb{H}=\left(\mathcal{H}_t:=\mathcal{F}^{\zeta,c}_{t}\vee \sigma(\tau),\,\, t\geq 0\right),
\end{equation}		
which is equal to the initial enlargement of the filtration $\mathbb{F}^{\zeta,c}$ by the $\sigma$-algebra $\sigma(\tau)$.
	Since the processes $\zeta$ and $Z$ are $\mathbb{H}$-adaped it follows from equation \eqref{RBrep} that $\widehat{N}$ is $\mathbb{H}$-adapted. Moreover, $\tau $ is a stopping time with respect to $\mathbb{H}$. The next proposition will play a very important role in forthcoming developments, since it shows the semimartingale property of $\zeta$ with respect to $\mathcal{H}_{t}$.
		 \begin{proposition}\label{propintYdecomsemiinL1}
	 	\begin{enumerate}
	 		\item[(i)] 	We have
	 		$$ \mathbb{E}\left[\int_{0}^{t}\,|Z_{s}|\,ds\right]<+\infty, \forall t\geq 0. $$
	 		\item[(ii)] The process $\widehat{ N}=(\widehat{ N}_{t},t\geq 0)$ defined  by
	 		\begin{equation}
	 		\widehat{N}_{t}=\zeta_{t}-\int_{0}^{t}\,Z_{s}ds,\,\, t\geq 0, \label{equationGmartingale}
	 		\end{equation}
	 		is a $\mathcal{H}_{t}$-martingale stopped at $\tau$.
	 	\end{enumerate}
	 \end{proposition}
	 \begin{proof}
	(i) We first note that $Z$ is a nonnegative process. Since, for $s\leq r$, $\zeta_s^{r}$ has a beta distribution $\beta (s,r-s)$ then $\mathbb{E}\left(\zeta_s^{r}\right)=s/r$.  So, we can
	 		see, for any $t\geq 0$, that 
	 		\begin{align*}
	 		\mathbb{E}\left[\int_{0}^{t} Z_{s}\, ds\right]&=\int_{0}^{+\infty}\int_{0}^{t\wedge r}\,\frac{1-\mathbb{E}[\zeta_s^{r}]}{r-s} \,ds \,\mathbb{P}_{\tau}(dr)\\
	 		&=\int_{0}^{+\infty} \,\int_{0}^{t\wedge r} \,\frac{1}{r} \,ds \,\mathbb{P}_{\tau}(dr)\leq 1.
	 		\end{align*}
	 			(ii) By assertion (i), the process $\left(Z_{t},t\geq 0\right)$ is
	 		integrable with respect to the Lebesgue measure, hence $\widehat{N}$ is well-defined.
	 		It is clear that the process $\widehat{N}$ is $\mathbb{H}$-adapted and $\widehat{N}_{t}=\widehat{N}_{\tau}$, $\mathbb{P}$-a.s, on the set $\{t\geq \tau \}$. Now since $(\widehat{N}_t^r, t\geq 0)$ is a $\mathbb{F}^{\zeta^{r}}$-martingale stopped at $r$ we obtain, for any $0<t_{1}<t_{2}<...<t_{n}=t$, $n\in \mathbb{N}^*$, $h \geq 0$ and $g$ a bounded Borel function, that
\begin{align*}\mathbb{E}\left[(\widehat{N}_{t+h}-\widehat{N}_{t})g(\zeta_{t_{1}},\ldots,\zeta_{t_{n}},\tau)\right] & =\int_{(0,+\infty)}\mathbb{E}[(\widehat{N}_{t+h}^{r}-\widehat{N}_{t}^{r})g(\zeta_{t_{1}}^{r},\ldots,\zeta_{t_{n}}^{r},r)]\mathbb{P}_{\tau}(dr)\\
\\
 & =\int_{(0,t)}\mathbb{E}[(\widehat{N}_{t+h}^{r}-\widehat{N}_{t}^{r})g(\zeta_{t_{1}}^{r},\ldots,\zeta_{t_{n}}^{r},r)]\mathbb{P}_{\tau}(dr)\\
 & +\int_{[t,t+h)}\mathbb{E}[(\widehat{N}_{t+h}^{r}-\widehat{N}_{t}^{r})g(\zeta_{t_{1}}^{r},\ldots,\zeta_{t_{n}}^{r},r)]\mathbb{P}_{\tau}(dr)\\
 & +\int_{[t+h,+\infty)}\mathbb{E}[(\widehat{N}_{t+h}^{r}-\widehat{N}_{t}^{r})g(\zeta_{t_{1}}^{r},\ldots,\zeta_{t_{n}}^{r},r)]\mathbb{P}_{\tau}(dr)\\
\\
 & =\int_{(0,t)}\mathbb{E}[(\widehat{N}_{r}^{r}-\widehat{N}_{r}^{r})g(\zeta_{t_{1}}^{r},\ldots,\zeta_{t_{n}}^{r},r)]\mathbb{P}_{\tau}(dr)\\
 & +\int_{[t,t+h)}\mathbb{E}[(\widehat{N}_{r}^{r}-\widehat{N}_{t}^{r})g(\zeta_{t_{1}}^{r},\ldots,\zeta_{t_{n}}^{r},r)]\mathbb{P}_{\tau}(dr)\\
 & +\int_{[t+h,+\infty)}\mathbb{E}[(\widehat{N}_{t+h}^{r}-\widehat{N}_{t}^{r})g(\zeta_{t_{1}}^{r},\ldots,\zeta_{t_{n}}^{r},r)]\mathbb{P}_{\tau}(dr)=0.
\end{align*}	The desired result follows by a standard monotone class argument. This completes the proof.
	 \end{proof}

Therefore, it follows from Stricker's Theorem \cite{ST} that $\zeta$ is a semimartingale relative to its natural filtration $\mathbb{F}^{\zeta,c}$. A natural question is: what is the explicit form of its canonical decomposition? That is the problem we want to discuss. The method consists in applying the stochastic fltering theory.
	 \begin{theorem}\label{thmdecompositionsemimartingale}
	 	 The canonical decomposition of $\zeta$ in its natural filtration $\mathbb{F}^{\zeta,c}$ is given by
	 	\begin{equation}
	 	\zeta_{t}=\widetilde{N}_t+\int_{0}^{t} \left(1-\zeta_{s}\right)\dint_{(s,+\infty)}\dfrac{1}{r-s}\phi_{\zeta_{s}^{r}}(\zeta_{s})\,\mathbb{P}_{\tau}(dr)\,\mathbb{I}_{\{0<\zeta_{s}<1\}}ds,\label{equationdecompositionsemi}
	 	\end{equation}
	 	where $(\widetilde{N}_t, t\geq0)$ is an $\mathbb{F}^{\zeta,c}$-martingale stopped at $\tau$.
	 \end{theorem}
	 \begin{proof}
Let us start by recalling that $\tau $ is a stopping time with respect to $\mathbb{F}^{\zeta,c}$. A well known result of filtering theory, \cite{BR} (T1, p. 87) (or Theorem 8.1.1 and Remark 8.1.1 \cite{KA} for more general setting) tells us that the decomposition of $\zeta$ in its natural filtration $\mathbb{F}^{\zeta,c}$ is given by
\begin{equation}
	 	\zeta_{t}=\widetilde{N}_t+\int_{0}^{t} \mathbb{E}\left(Z_s\vert \mathcal{F}_{s}^{\zeta,c}\right) ds,
	 	\end{equation}
	 	where $(\widetilde{N}_t, t\geq0)$ is an $\mathbb{F}^{\zeta,c}$-martingale stopped at $\tau$.
Therefore, we have only to compute the conditional expectation of $Z_s$ relative to $\mathcal{F}_s^{\zeta,c}$. Indeed, using  \eqref{equationtaugivenF} we have 
\[
\begin{array}{ll}
\mathbb{E}\left(Z_{s}\vert\mathcal{F}_{s}^{\zeta,c}\right) & =\mathbb{E}\left(\dfrac{1-\zeta_{s}}{\tau-s}\,\mathbb{I}_{\{s<\tau\}}\vert\mathcal{F}_{s}^{\zeta,c}\right)=\left(1-\zeta_{s}\right)\mathbb{E}\left(\dfrac{1}{\tau-s}\,\mathbb{I}_{\{s<\tau\}}\vert\mathcal{F}_{s}^{\zeta,c}\right)\\
\\
 & =\left(1-\zeta_{s}\right)\dint_{(s,+\infty)}\dfrac{1}{r-s}\phi_{\zeta_{s}^{r}}(\zeta_{s})\,\mathbb{P}_{\tau}(dr)\,\mathbb{I}_{\{0<\zeta_{s}<1\}}.
\end{array}
\]
	Hence we derive the canonical decomposition \eqref{equationdecompositionsemi} of $\zeta$ as a semimartingale in its
own filtration $\mathbb{F}^{\zeta,c}$. \end{proof}
\begin{remark}
The results of the paper can be straightforwardly extended to a large
class of gamma subordinator $(\gamma_{t}^{(\eta,\kappa)},\,t\geq0)$,
$\eta,\kappa>0$, with Lévy measure 
\[
\nu(dx)=\dfrac{\kappa}{x}\,\exp(-\eta x)\,\mathbb{I}_{\left(0,\infty\right)}(x)\,dx
\]
and whose law at time $t$ is the gamma distribution with density
\[
f_{\gamma_{t}^{(\eta,\kappa)}}(x)=\dfrac{\eta^{\kappa t}\,x^{\kappa t-1}\,\exp(-\eta x)}{\Gamma(\kappa t)}\mathbb{I}_{\left(0,\infty\right)}(x).
\]
The Lévy-Khintchine representation is given by
\[
\mathbb{E}(\exp(-\lambda\gamma_{t}^{(\eta,\kappa)}))=(1+\dfrac{\lambda}{\eta})^{-\kappa t}.
\]
On the other hand, they can be also easily extended to the gamma bridges
of length $r$, starting at $0$, with an arbitrary ending point $a>0$
\[
\zeta_{t}^{r}:=a\dfrac{\gamma_{t\wedge r}^{(m)}}{\gamma_{r}^{(m)}},~t\geq0.
\]
 For the sake of simplicity, we have therefore considered only the case
$\eta=\kappa=a=1$ without loss of generality.

\end{remark}

\end{document}